\documentclass[reqno,12pt,letterpaper]{amsart}
\usepackage[proof]{newmzmacros2}
\usepackage[all,cmtip]{xy}

\usepackage{soul}

\def\red#1{\textcolor{red}{#1}}
\def\blue#1{\textcolor{blue}{#1}}

\title{An abstract formulation of the flat band condition}

\author{Jeffrey Galkowski}
\email{j.galkowski@ucl.ac.uk}
\address{Department of Mathematics, University College London, WC1H 0AY, UK}

\author{Maciej Zworski} 
\email{zworski@math.berkeley.edu}
\address{Department of Mathematics, University of California, Berkeley, CA 94720}

\begin{document}

\begin{abstract}
Motivated by the study of flat bands in models of twisted bilayer graphene (TBG), we give abstract conditions which guarantee the existence of a discrete set of parameters for which periodic Hamiltonians exhibit flat bands. As an application, we show that a scalar operator derived from the chiral model of TBG has flat bands for a discrete set of parameters. \end{abstract}

\maketitle

\section{Introduction}

Existence of flat bands for periodic operators (in the sense of Floquet theory) has interesting
physical consequences, especially in the case of nontrivial band topology. A celebrated
recent example is given by the Bistritzer--MacDonald Hamiltonian \cite{BM11} modeling twisted
bilayer graphene 
(see \cite{CGG} and \cite{wats} for its mathematical derivation).  A model exhibiting exact flat bands is given by the chiral limit of the Bistritzer--MacDonald
model considered by Tarnopolsky--Kruchkov--Vishwanath \cite{magic}. Both the Bistritzer--MacDonald model and its chiral limit depend on a parameter corresponding to the angle of twisting between two graphene sheets and, in the chiral model, the perfectly flat bands
appear for a discrete set of values of this parameter.
This follows from a spectral characterization of those {\em magic} angles
given by Becker--Embree--Wittsten--Zworski \cite{beta}. Existence of the first real magic angle
was provided by Watson--Luskin \cite{lawa}, with its simplicity established by 
Becker--Humbert--Zworski \cite{bhz1}. That paper also showed existence of infinitely many,
possibly complex, magic angles.

The purpose of this note is to provide a simple abstract version of the spectral 
characterization of magic angles given in \cite{beta} (see also \cite[Proposition 2.2]{bhz2}). In \S \ref{s:scal} we 
apply this spectral characterization of flat bands in a model to which the argument from \cite{beta} does not apply.

To formulate our result we consider Banach spaces, 
 $X\subset Y$,  and a  connected open set $ \Omega  \subset \mathbb C $. 
   The result concerns a holomorphic family of {\em Fredholm operators 
 of index $ 0 $}
 (see \cite[\S C.2]{res}):
 \begin{equation} 
\label{eq:defQ}
Q : \Omega \times \mathbb C \to \mathcal L ( X, Y ) , \ \ 
( \alpha , k ) \mapsto Q ( \alpha, k ) . 
\end{equation}
We make the following assumption: there exists a lattice
 $ \Gamma^* \subset \mathbb C $, 
and families of {\em invertible} operators $ \gamma \mapsto W_\bullet (\gamma ) : \bullet \to \bullet $, 
$ \bullet = X, Y $, $ \gamma \in \Gamma^* $, 
such that 
\begin{equation}
\label{eq:lattice} 
\begin{gathered}
Q ( \alpha , k + \gamma ) = W_Y ( \gamma )^{-1} Q ( \alpha, k ) W_X ( \gamma ) , \ \ 
\gamma \in \Gamma^*  . 
\end{gathered} 
\end{equation} 

A guiding example is given by the chiral model of twisted bilayer graphene (TBG) \cite{magic}, \cite{beta}, 
\cite{bhz2}:
\begin{equation}
\label{eq:defD}   
\begin{gathered} Q ( \alpha, k ) := D ( \alpha ) + k , \ \  
D ( \alpha ) := \begin{pmatrix} 2 D_{ \bar z } & \alpha U ( z ) \\
\alpha U ( - z ) & 2 D_{\bar z } \end{pmatrix} ,
\ \ \ \Omega  = \mathbb C , 
 \\  2D_{\bar z } = \tfrac 1 i ( \partial_{{x_1}} + i \partial_{x_2} ) , \ \
z = {x_1} + i x_2 \in \mathbb C , 
\end{gathered}
\end{equation}
where $ U $ satisfies 
\begin{equation}
\label{eq:defU}
\begin{gathered}   U ( z + \gamma ) = e^{ i \langle \gamma , K  \rangle } U ( z ) , \ \ U ( \omega z ) = \omega U(z) , \ \ \overline {U ( \bar z ) } = - U ( - z ) , \ \ \ 
\omega = e^{ 2 \pi i/3},  \ \\ \gamma \in \Lambda := \omega \mathbb Z \oplus \mathbb Z , \ \ 
\omega K \equiv K \not \equiv 0 \!\!\! \mod \Lambda^* , \ \ 
\Lambda^* :=  \frac {4 \pi i}  {\sqrt 3}  \Lambda , \ \ \langle z , w \rangle := \Re ( z \bar w ) . 
\end{gathered} 
\end{equation}
An example of $ U $  is given by the Bistritzer--MacDonald potential
\begin{equation}
\label{eq:defU2}
U ( z ) =  - \tfrac{4} 3 \pi i \sum_{ \ell = 0 }^2 \omega^\ell e^{ i \langle z , \omega^\ell K \rangle }, \ \ \ K = \tfrac43 \pi  .
\end{equation}
We note that a potential satisfying \eqref{eq:defU} is periodic with respect to the lattice $ 3 \Lambda $
and that we can take
\[    Y := L^2 ( \mathbb C / \Gamma ; \mathbb C^2 ) , \ \ \ 
 X := H^1 ( \mathbb C / \Gamma ; \mathbb C^2 ) , \ \ \  \Gamma := 3 \Lambda . \]
(For the Fredholm property of $ D ( \alpha ) + k : X \to Y $ see 
\cite[Proposition 2.3]{beta}; the index is equal to $ 0 $.) The operators $ W_\bullet ( \gamma ) $
are given by multiplication by $ e^{ i \langle \gamma, z \rangle } $, $ \gamma \in \Gamma^* $,
with $\Gamma^*$ the dual lattice to $ \Gamma $. (The operator is the same but acts on different spaces.)

The self-adjoint Hamiltonian for the chiral model of TBG is given by
\begin{equation}
\label{eq:defH}   H ( \alpha ) := \begin{pmatrix} 0 &  D( \alpha )^* \\
D ( \alpha ) & 0 \end{pmatrix} , \end{equation}
and Bloch--Floquet theory means considering the spectrum of
\begin{equation}
\label{eq:defHk} 
\begin{gathered}
 H_k ( \alpha ) := e^{ - i \langle z, k \rangle } H ( \alpha ) e^{ i \langle z, k \rangle } 
: H^1 ( \mathbb C/\Gamma ; \mathbb C^4 ) \to L^2 ( \mathbb C/\Gamma ; \mathbb C^4 ) , \\
H_k ( \alpha ) = \begin{pmatrix} 0  & Q ( \alpha, k )^* \\
Q ( \alpha, k ) & 0 \end{pmatrix} , \ \ \ Q ( \alpha, k ) = D ( \alpha ) + k ,
\end{gathered}
\end{equation}
see \cite{beta} (we should stress that it is better to consider a modified boundary 
condition \cite{bhz2} rather than $ \Gamma$-periodicity but this plays no role in the discussion here). 

A {\em flat band} at zero energy for the Hamiltonian \eqref{eq:defH} means that 
\begin{equation}
\label{eq:flat}  
\begin{split}  \forall \,  k \in \mathbb C  \ \ 0 \in \Spec_{ L^2 ( \mathbb C/\Gamma; \mathbb C^4 ) } H_k ( \alpha ) 
 & \ \Longleftrightarrow \ \forall \, k \in \mathbb C \ \ 
\ker_{ H^1 ( \mathbb C/\Gamma; \mathbb C^4 )}  H_k ( \alpha ) \neq \{ 0 \} \\
& \ \Longleftrightarrow \ \forall \, k \in \mathbb C \ \ 
\ker_{ H^1 ( \mathbb C/\Gamma; \mathbb C^2 )}  Q  ( \alpha, k ) \neq \{ 0 \} . \end{split} 
\end{equation}

We generalize the result of \cite{beta} stating that the set of $ \alpha$'s for which
\eqref{eq:flat} holds, which we denote by $\mathcal{A}_{\rm{ch}}$, is a discrete subset of $ \mathbb C $ and that \eqref{eq:flat} is
equivalent to 
\begin{equation}
\label{eq:flat1}  \ \exists \, k \in \mathbb C \setminus \Gamma^* \ \ 
\ker_{ H^1 ( \mathbb C/\Gamma; \mathbb C^2 )}  Q  ( \alpha, k ) \neq \{ 0 \} . \end{equation}
The key property in showing this is the existence of protected states \cite{magic}, \cite{beta}: 
\begin{equation}
\label{eq:prot}
\forall \,  \alpha \in \mathbb C , \ 
k \in \Gamma^* \ \  \dim \ker_{ H^1 ( \mathbb C/\Gamma; \mathbb C^2 )}  Q  ( k,  \alpha ) 
\geq 2, \ \ \    \dim \ker_{ H^1 ( \mathbb C/\Gamma; \mathbb C^2 )}  Q  ( k , 0 ) = 2 . 
 \end{equation}
In the treatment of abstract operators \eqref{eq:defQ} we cannot consider 
 the kernel alone as there are many degenerate possibilities -- see \cite[\S C.4]{res}
for a review of the {\em Gohberg--Sigal} theory relevant to this. 
In \eqref{eq:prot} semisimplicity of the spectrum of $ D ( 0 ) $, 
$ Q ( k, 0 ) = D ( 0 ) + k $ was implicit but this does not need to hold for more general
operators, including more general $ D (\alpha ) $ -- see for instance \cite{yang}. Hence we will replace
\eqref{eq:prot} by a different hypothesis (see \eqref{eq:hyp} below) which involves
a more general notion of multiplicity:
We define multiplicity as follows:  if for $ \alpha \in \Omega $,  there is $k_0\in \mathbb{C}$ such that $\ker Q(\alpha, k_0)=\{0\}$, then we define 
 \begin{equation}
\label{eq:mult}
m (\alpha, k ) := \frac{1}{ 2 \pi i } \tr \oint_{\partial D }   Q( \alpha, \zeta ) ^{-1} 
\partial_\zeta Q ( \alpha , \zeta ) d \zeta , 
\end{equation} 
where the integral is over the positively oriented boundary of a disc 
$ D $ which contains $ k$ as the
only possible pole of $ \zeta \mapsto Q ( \alpha, \zeta ) $. 
Otherwise, we  put $ m ( \alpha, k ) =  
\infty $ for all $k\in \mathbb{C}$.

\noindent{\bf Remarks.} 
1. Since $k\mapsto Q(\alpha,k)$ is a holomorphic family of Fredholm operators with index $0$, we observe that  that $\ker Q(\alpha, k_0)=\{0\}$ implies  existence of $Q(\alpha,k_0)^{-1}$ and hence $Q(\alpha,\zeta)^{-1}$ is a meromorphic family of operators -- see
\cite[Theorem C.8]{res}. In particular,~\eqref{eq:mult} is well-defined. Hence we have a dichotomy:
for a fixed $ \alpha $
\begin{equation}
\label{eq:dich}   \forall  \, k  \ \ m ( \alpha , k ) < \infty  \ \ \text{ or } \ \
\forall \, k \ \ m ( \alpha , k ) = \infty . 
\end{equation}

\noindent
2. Assumption~\eqref{eq:lattice} implies that $m(\alpha,k)=m(\alpha,k+\gamma)$ for any $\gamma \in \Gamma^*$.

\noindent
3.  For any $\alpha\in \Omega$ and $ k_1 \in \mathbb{C}$, we have $m(\alpha, k_1 )\geq \dim \ker Q(\alpha, k_1 )$. If $ Q ( \alpha , k ) = P ( \alpha ) + k $,  $ -k_1 $ is a semisimple eigenvalue
of $ P ( \alpha ) $,  and  $m(\alpha,k_1 )<\infty$,
then $m(\alpha,k_1 )=\dim \ker Q(\alpha,k_1 )$. 

\begin{theo}
\label{t:1}  In the notation of \eqref{eq:defQ} and assuming \eqref{eq:lattice} 
suppose that for some  $ \alpha_0 \in \Omega  $ and every $ k \in \mathbb C $, we
have,  
\begin{equation}
\label{eq:hyp} 
\begin{gathered}
m ( \alpha, k  ) \geq m ( \alpha_0, k ) \neq \infty .
\end{gathered} 
\end{equation}
Then there exists a discrete set $ \mathcal A\subset \Omega  $ such that for all $ k \in \mathbb C $
\begin{equation}
\label{eq:magic}   \begin{gathered}
m ( \alpha, k ) = \left\{ \begin{array}{ll} 
\ \ \infty &  \alpha \in \mathcal A, \\
m( \alpha_0, k ) &  \alpha \notin \mathcal A . \end{array} \right. 
\end{gathered}
\end{equation}
\end{theo} 

In view of \eqref{eq:prot} (and semisimplicity -- see Remark 3 above) we see that \eqref{eq:hyp} is satisfied for $ Q $ given in 
\eqref{eq:defD} with 
$ \alpha_0 = 0 $,  $ \Omega  =\mathbb C $ and 
$ m ( 0, k ) = 2 \indic_{ \Gamma^* } ) (k ) $.
 For a direct proof see \cite[\S 3]{beta} or \cite[\S 2]{bhz2}.

\noindent
{\bf Remark.} 
Theorem~\ref{t:1} is valid under a weaker condition than~\eqref{eq:lattice}. 
As seen in \S \ref{s:pr}, we need to control the multiplicity $m(\alpha,k)$ for every $k$
using $m(\alpha,k)$ for $ k $ in some fixed compact set. 
 That some condition is needed (other than holomorphy and the Fredholm property) 
 can be seen by considering the simple example of $Q(\alpha,k)=1-\alpha k$, 
 $ X = Y = \mathbb C $. In this case~\eqref{eq:hyp} is satisfied with $\alpha_0=0$ and $m(\alpha_0,k)=0$. Nevertheless, 
$$
m(\alpha,k)\geq \dim \ker Q(\alpha,k) =\begin{cases}0&k\neq \alpha^{-1}\\ 1&k=\alpha^{-1}\end{cases}
$$ 
and \eqref{eq:magic} fails. We opted for the easy to state condition \eqref{eq:lattice} in view
of the motivation from condensed matter physics.


\section{Proof of Theorem \ref{t:1}}
\label{s:pr}

Define $\mathcal{K}:=\supp m(\alpha_0,\bullet)$, which is 
a discrete set (see Remark 1 above).
We now fix $ k_0 \in \mathbb{C}\setminus \mathcal K$ and define 
\begin{equation}
\label{eq:defAk0} \mathcal A_{k_0}  := \complement \{ \alpha  \in \Omega  :  Q(\alpha, k_0  )^{-1}:Y\to X \text{ exists} \}. \end{equation}
Since $ \alpha \mapsto Q( \alpha , k_0 ) $ is a holomorphic family of Fredholm operators
of index zero, 
and $ \ker  Q( \alpha_0, k_0  ) = \{ 0 \}  $, we conclude that $ \alpha \mapsto Q ( \alpha , k_0 )^{-1} $ is 
a meromorphic family of operators and, in particular, $ \mathcal A_{k_0}  $ is a discrete set --
see \cite[\S C.3]{res}. Also, for 
$ \alpha \notin \mathcal A_{k_0}  $, $ k \mapsto Q ( \alpha, k )^{-1} $ is a meromorphic
family of operators and hence $m(\alpha,k)<\infty$ for all $k\in \mathbb{C}$
For $ D $ as in~\eqref{eq:mult},
there exists $ \varepsilon > 0 $  such that 
\begin{equation}
\label{eq:mult2}
m (  \alpha,k ) = \sum_{ k' \in D } m(  \alpha',k' ) , \ \text{ if $ |\alpha - \alpha' | < \varepsilon $.} 
\end{equation}
In particular for a fixed $ k \in \mathbb{C} $, $ \alpha \mapsto m (  \alpha ,k ) $ is upper semicontinuous.
We now define
\[ U := \{ \alpha \in \Omega \setminus \mathcal A_{k_0}  : \forall \, k,  \  m ( \alpha,k ) = m(\alpha_0,k) \}. \]
We note that
$ \alpha_0\in U  $ and that $ \Omega \setminus \mathcal A_{k_0}  $ is connected. Hence $ U  = \Omega \setminus 
\mathcal A_{k_0}  $ if we show that $U $ is open and closed in the relative topology of 
$ \Omega \setminus  \mathcal A_{k_0}  $.

Let $\alpha \in U$. We start by showing that for any compact subset $K\subset \mathbb{C}$, there 
exists $ \varepsilon_K>0$ such that 
\begin{equation}
\label{e:compactOpen}
m(\alpha',k)= m(\alpha_0,k)=m(\alpha,k)\text{ for all }k\in K\text{ and }|\alpha-\alpha'|<\varepsilon_K.
\end{equation} 
To see this we note that for any fixed $ k \in \mathbb C $
 there exist $D_k = D ( k, \delta_k ) $,   
     and $\varepsilon_k>0$ such that that~\eqref{eq:mult2} holds for $|\alpha-\alpha'|<\varepsilon_k$.   
By shrinking $ D_k $ (and consequently $ \varepsilon_k$) we can assume that (here we use
the discreteness of $ \mathcal K $) 
\begin{equation}
\label{eq:Dkset}      D_k \setminus \{ k \}  \subset \complement \mathcal K. 
\end{equation}
Since $ K $ is compact, we can find a finite cover
$
K\subset \bigcup_{i=1}^N D_{k_i} $. 
Then $k_i$ is the only possible pole for $k\mapsto Q(\alpha,k)^{-1}$ in $D_{k_i}$ 
and for $|\alpha-\alpha'|<\varepsilon_K:=\min_{i=1,\dots N}\varepsilon_{k_i}$, we have
$$
m(\alpha,k_i)=\sum_{k\in D_{k_i}}m(\alpha',k),
$$
and in particular, as $\alpha\in U$,  
$
m(\alpha_0,k_i)=\sum_{k\in \mathcal{D}_{k_i}}m(\alpha',k).
$ 
Since $m(\alpha',k_i)\geq m(\alpha_0,k_i)$ (by the assumption \eqref{eq:hyp}) and $m(\alpha',k)\geq 0$ for all $k$ (since $k\mapsto Q(\alpha,k)$ is holomorphic), we have 
$$
0\geq m(\alpha_0,k_i)-m(\alpha',k_i)=\sum_{k\in D_{k_i}\setminus \{k_i\}}m(\alpha',k)\geq 0.
$$
Hence, $m(\alpha_0,k_i)=m(\alpha',k_i)$ and $m(\alpha',k)=0$ for $k\in D_{k_i}\setminus \{k_i\}$. In particular, $m(\alpha',k)=m(\alpha_0,k)$ for all $k\in K$ and $|\alpha'-\alpha|<\varepsilon_K$ as claimed in \eqref{e:compactOpen}.


Now, to complete the proof that $U$ is open, we use~\eqref{eq:lattice}. Let $K\subset \mathbb{C}$ contain the fundamental domain of $\Gamma^*$ and $\varepsilon_K$ as in~\eqref{e:compactOpen}. Then, for all $k\in \mathbb{C}$, there is $\gamma\in \Gamma^*$ such that $k+\gamma\in K$. Using~\eqref{e:compactOpen}, we have for $|\alpha-\alpha'|<\varepsilon_K$, 
$$
m(\alpha',k+\gamma)=m(\alpha,k+\gamma).
$$
But then, by~\eqref{eq:lattice}
$
m(\alpha',k+\gamma)=m(\alpha',k)$, $ m(\alpha,k+\gamma)=m(\alpha,k)
$,
and hence
$$
m(\alpha',k)=m(\alpha,k)=m(\alpha_0,k).
$$
Since $k\in \mathbb{C}$ was arbitrary, this implies $\alpha'\in U$.

To show that $ U $ is closed suppose that  $ \mathcal A_{k_0} \not \ni \alpha_j \to \alpha \notin \mathcal A_{k_0 } $
and $ m ( \alpha, k_j ) = m(\alpha_0,k) $. Then, since $\alpha\notin \mathcal{A}_{k_0}$, for every $k\in \mathbb{C}$, there exist $\varepsilon_k>0$ and $ D_k $ 
 such that~\eqref{eq:mult2} and \eqref{eq:Dkset} hold. In particular, for $j$ large enough (depending on $k$),
$$
m(\alpha,k)=\sum_{k'\in D_k }m(\alpha_j,k')=\sum_{k'\in D_k}m(\alpha_0,k')=m(\alpha_0,k).
$$
Hence $ U $ is closed and open which means that $ U = \Omega \setminus \mathcal A_{k_0} $.

Recalling the definition \eqref{eq:defAk0},  we proved that
\[ \Omega \setminus \mathcal A_{k_0} \subset \{ \alpha: \forall \, k,  \   m ( \alpha,k ) = m(\alpha_0,k) \} 
\subset \Omega  \setminus \mathcal A_{ k_1 } , \]
for any $ k_1 \notin \mathcal K  $. But this means that $ \mathcal A_{k_0 } $ is independent of $ k_0 $
and for $ \alpha \in \mathcal A := \mathcal A_{k_0} $, $ Q ( \alpha, k )^{-1} $ does not exist for any $ k 
\in \mathbb{C} $. In particular, $m(\alpha,k)=\infty$ for $\alpha\in \mathcal{A}$ and $k\in \mathbb{C}$.
\qed

\section{A scalar model for flat bands}
\label{s:scal}

One of the difficulties of dealing with the model described by \eqref{eq:defD}, \eqref{eq:defH}
is the fact that $ D ( \alpha ) $ acts on $ \mathbb C^2 $-valued functions. Here we propose the following
model in which $ D ( \alpha ) $ is replaced by a scalar (albeit second order) operator. This is done
as follows. We first consider $ P ( \alpha ) : H^2 ( \mathbb C/\Gamma ; \mathbb C^2 ) 
\to L^2 ( \mathbb C/\Gamma ; \mathbb C^2 ) $ defined as follows:
\begin{equation}
\label{eq:defP}
\begin{gathered} 
P ( \alpha ) := D ( - \alpha ) D ( \alpha ) = Q ( \alpha )  \otimes I_{\mathbb C^2} + R ( \alpha ), \ \ \ 
Q ( \alpha ): = ( 2 D_{\bar z } )^2 - \alpha^2 V ( z ) ,  \\
R ( \alpha ) :=  -  \alpha \begin{pmatrix} \ \ 0 &  V_1 ( z )  \\
V_1 ( -z ) & 0   \end{pmatrix} , \ \  V ( z ) := U ( z ) U( -z ) , \ \ V_1 ( z ) := 2 D_{\bar z } U ( z ) .
\end{gathered} 
\end{equation}
If we think of $ P ( \alpha ) $ as a semiclassical differential system with $ h = 1/\alpha $
(see \cite[\S E.1.1]{res}) then $ Q ( \alpha ) $ is the quantization of the determinant of 
the symbol of $ D( \alpha )$ and $ R ( \alpha ) $ is a lower order term.
We lose no information when considering $ P ( \alpha ) $ in the characterization 
of flat bands \eqref{eq:flat}:
\begin{prop}
\label{p:info}
If $  P ( \alpha, k ) := e^{ - i \langle z, k \rangle } P  ( \alpha ) e^{ i \langle z, k \rangle } $
then 
\begin{equation}
\label{eq:P2D}   \ker_{ H^1 ( \mathbb C/\Gamma ) } ( D ( \alpha ) + k ) \neq \{0 \}
\ \Longleftrightarrow \ \ker_{ H^2 ( \mathbb C/\Gamma ) } P ( \alpha, k ) \neq \{ 0 \} . 
\end{equation}
In particular $ \alpha \in \mathcal A_{\rm{ch}} $ if and only if 
$ k  \in \Spec_{ L^2 ( \mathbb C / \Gamma ) } P ( \alpha, k ) $ for some
$ k \notin \Gamma^* $ (which then implies this for all $ k $). 
\end{prop}
\begin{proof} 
We note that $ P ( \alpha , k ) = ( D ( - \alpha ) + k ) ( D ( \alpha)  + k ) $ and that 
\[   D ( - \alpha ) - k  = - \mathscr R ( D ( \alpha ) + k ) \mathscr R,
 \ \ \ \mathscr R \begin{pmatrix}
u_1 \\ u_2 \end{pmatrix} ( z ) = \begin{pmatrix} u_2 ( -z ) \\
u_1 ( - z ) \end{pmatrix} \]
and hence 
\[  \ker_{ H^1 ( \mathbb C/\Gamma ) }  ( D ( \alpha ) + k ) = 
\mathscr R \ker_{ H^1 ( \mathbb C/\Gamma ) }  ( D ( - \alpha ) - k ) . \]
Since $ D ( \alpha ) $ is elliptic, the elements of the kernels above are in $ C^\infty 
( \mathbb C/\Gamma ) $ and hence $ H^1 $ can be replaced by $ H^s $ for any $ s $
-- see \cite[Theorem 3.33]{res}. Hence if $ \ker_{ H^2 } P ( \alpha, k ) \neq\{ 0 \} $
then either $ \ker_{ H^2 }  ( D ( \alpha ) + k ) = \ker_{H^1 } ( D ( \alpha)  + k )  \neq 
\{ 0 \} $ or $ \ker_{ H^1 }  ( D ( - \alpha ) + k )  \neq 
\{ 0 \} $. If $ k \notin \Gamma^* $ then the equivalence of \eqref{eq:flat1} and 
 \eqref{eq:flat} gives the conclusion. 
\end{proof}

We now consider a model in which we drop the matrix terms in 
\eqref{eq:defQ}, the definition of $ P ( \alpha ) $, and have $ Q ( \alpha ) 
$ act on scalar valued functions. The self-adjoint Hamiltonian corresponding to 
\eqref{eq:defH} is now given by 
\begin{equation}
\label{eq:defHV}
\begin{gathered}
H ( \alpha )  := \begin{pmatrix} 0 & Q ( \alpha )^*  \\
Q ( \alpha ) & 0 \end{pmatrix}, \ \ \ Q ( \alpha ): = ( 2 D_{\bar z } )^2 - \alpha^2 V ( z ) ,  \ \ V \in C^\infty ( \mathbb C ) , \\
  V ( x + \gamma ) = V ( x ) , \ \ \gamma \in \Lambda:= 
\omega \mathbb Z \oplus \mathbb Z , \ \  V ( \omega x ) = \bar \omega V ( x ) , \ \
\omega := e^{ 2 \pi i/3 } . 
\end{gathered}
\end{equation}
The potential is periodic with respect to $ \Lambda $, and hence the usual Floquet theory 
applies:
\begin{equation}
\label{eq:Floquet}
\begin{gathered}
H( \alpha , k ) := \begin{pmatrix} 0 & Q( \alpha , k )^*  \\
Q ( \alpha, k  ) & 0 \end{pmatrix}, \ \ \ Q  ( \alpha , k ): = ( 2 D_{\bar z } + k )^2 - \alpha^2 V ( z ) , \\
\Spec_{L^2 ( \mathbb C ) }  H ( \alpha )  = \bigcup_{ k \in \mathbb C/\Lambda^* } \Spec_{ L^2 ( 
\mathbb C / \Lambda )} H ( \alpha, k  ) , 
\end{gathered}
\end{equation}
where $ \Spec_{ L^2 ( \mathbb C / \Lambda )} H ( \alpha , k ) $ is discrete and 
is symmetric under  $ E \mapsto - E $. Just as for the chiral model of 
TBG, a flat band at zero for a given $ \alpha$ means that 
\[ \forall \, k \in \mathbb C \  \ 0 \in \Spec_{ L^2 ( 
\mathbb C / \Lambda ; \mathbb C^2 )} H ( \alpha , k )  \ \Longleftrightarrow 
\  \forall \, k \in \mathbb C \  \ker_{ H^2 ( 
\mathbb C / \Lambda; \mathbb C  )} Q ( \alpha, k  ) \neq \{ 0 \} .
 \]
As in the chiral model, we take $W_X(\gamma)=W_Y(\gamma)=e^{i\langle \gamma,z\rangle}$, $\gamma\in \Lambda^*$, the dual lattice to obtain~\eqref{eq:lattice}. 
Theorem \ref{t:1} shows that as in the case of \eqref{eq:defH} this happens
for a discrete set of $ \alpha \in \mathbb C $:

\begin{theo}
\label{t:2} 
For $ H $ and $ Q $ given in \eqref{eq:defHV} there exists a discrete set 
$ \mathcal A_{\rm{sc}}  \subset \mathbb C $ such that 
\begin{equation}
\label{eq:t2}
\begin{gathered}
\ker_{ H^2 ( \mathbb C / \Lambda; \mathbb C  )} Q ( \alpha, k ) \neq \{ 0 \} \  \text{ for 
$ \alpha \in \mathcal A_{\rm{sc}} $, $ k \in \mathbb C $, } \\
m(\alpha,k) = 2
\indic_{ \Lambda^*} ( k )  \ \text{ for $ \alpha \notin \mathcal A_{\rm{sc}}   $.}  
\end{gathered}
\end{equation}
\end{theo}
 
This is an immediate consequence of Theorem \ref{t:2} once we establish
\eqref{eq:hyp} with $m(0,k) =2 \indic_{\Lambda^*}(k)$  (i.e. $ \alpha_0 = 0 $). The kernel of
$ Q ( 0 , k ) = 2 ( D_{\bar z } + k )^2 $ , 
on $ H^2 ( \mathbb C/\Lambda ) $ is empty for $ k \notin \Lambda^* $ and
is given by $ \mathbb C e^{ i \langle k , z \rangle } $, when $ k \in \Lambda^* $. 
This gives that $m(0,k)=2\indic_{\Lambda^*}(k)$ after noticing that $D_{\bar{z}}$ is diagonal in the basis $ \{e^{ i \langle k , z \rangle }\}_{k\in \Lambda^*}$ of $H^2(\mathbb{C}/\Lambda;\mathbb{C})$ and $\partial_k(D_{\bar{z}}+k)^2|_{k=k_0} e^{ i \langle k_0 , z \rangle }=0$ for $k_0\in \Lambda^*$.  The second one is provided by 
\begin{prop}
\label{p:prot}
For all $ \alpha \in \mathbb C $  and $ k \in \Lambda^* $, 
$m(\alpha,k)\geq 2$.
\end{prop}

\begin{proof}
The proof uses the symmetry of $ Q ( \alpha, k ) $ under the action $ k \mapsto \omega k $
in a way similar to its use in \cite{bz1} and \cite{BHWY}. 

We first recall that
\eqref{eq:lattice} implies that $ m ( \alpha, k + \gamma ) = m ( \alpha, k ) $ for $ \gamma \in 
\Gamma^* $ and hence it is enough to show that $ m ( \alpha, 0 ) \geq 2 $ for all $
\alpha $. We then define 
$$
\mathcal{C}:=\{\alpha \in \mathbb{C}\,:\, m(\alpha,k)<\infty,\,\text{ for all }k\in \mathbb{C}\}
= \{ \alpha \, : \, m ( \alpha, k_0 ) < \infty \} ,
$$
where the second equality holds, in view of \eqref{eq:dich}, for any $ k_0 \in \mathbb C $. 
This set is connected, as for $ k_0 \notin \Lambda^* $, $ Q ( \alpha, k_0 )^{-1} $ exists
and hence $ \alpha \mapsto Q ( \alpha, k_0 )^{-1} $ is a meromorphic family of
operators (we use \cite[Theorem C.8]{res} again).

Next, we observe that, 
$$
Q (\alpha,k)\Omega=\omega \Omega   Q(\alpha,\omega k), \ \ \
\Omega u (z):=u(\omega z), 
$$
and that gives,  for $\alpha\in\mathcal{C}$, 
\begin{equation}
\label{eq:omk}
m(\alpha,k)=m(\alpha,\omega k)=m(\alpha,\omega^2k). 
\end{equation}
We now let
$$
\mathcal{B}:=\{ \alpha\in\mathcal{C}\,:\, m(\alpha,0)=2 \!\!\! \mod \! 3\}.
$$
and claim that $\mathcal{B}= \mathcal{C}$. This will finish the proof since $m(\alpha,k)\geq 0$ implies that 
$
\mathcal{B}\subset \{ \alpha\in\mathcal{C}\,:\, m(\alpha,0)\geq 2\}.
$
Since $0\in\mathcal{B}$ and $\mathcal{C}$ is connected, to show that $\mathcal{B}=\mathcal{C}$, it is enough to show that $\mathcal{B}\subset\mathcal{C}$ is open and closed in the relative topology
of $\mathcal C$.

We start by showing that $\mathcal{B}$ is open and for that choose $\alpha_0\in \mathcal{B}$. Then,
in view of \eqref{eq:mult2} and \eqref{eq:omk}, there exists a disk,  $D$,  around $0$ and $\varepsilon >0$ such that for $|\alpha-\alpha_0|<\varepsilon$, 
\begin{align*}
2=m(\alpha_0,0)&=\sum_{k\in D}m(\alpha,k)
=m(\alpha,0)+\sum_{k\in D\setminus \{0\}}m(\alpha,k)
\\ &
=m(\alpha,0)+3\sum_{\substack{k\in D\setminus \{0\}\\ 0\leq \arg k<\frac{2\pi}{3}}}m(\alpha,k).
\end{align*}
It follows that $m(\alpha,0)=2\!\!\! \mod \! 3$ which implies that 
$\alpha\in\mathcal{B}$ for $|\alpha-\alpha_0|< \varepsilon$, that is,   $\mathcal{B}$ is open as claimed. 

Next, we show $\mathcal{B}$ is closed. To see this, suppose that $\alpha_j\in \mathcal{B}$ with $\alpha_j\to \alpha\in \mathcal{C}$. Then, for $j$ large enough, \eqref{eq:mult2} gives
\begin{align*}
m(\alpha,0)&=\sum_{k\in D}m(\alpha_j,k) 
=m(\alpha_j,0)+\sum_{k\in D\setminus \{0\}}m(\alpha_j,k)
=2+3\sum_{\substack{k\in D\setminus \{0\}\\ 0\leq \arg k<\frac{2\pi}{3}}}m(\alpha_j,k).
\end{align*}
Hence, $m(\alpha,0)=2\!\!\!\!\ \mod \! 3$, that is $\mathcal{B}$ is closed. 
\end{proof}

\noindent
{\bf Remarks.} 1. The proof of Theorem \ref{t:1} 
also shows the following spectral characterization of $\mathcal A_{\rm{sc}} $: if
\begin{equation}
\label{eq:defT}
T_k :=  ( 2 D_{\bar z } + k)^{-2} V , \ \ \ k \notin \Lambda^* , 
\end{equation}
then 
\begin{equation}
\label{eq:spectral}
\begin{split} \alpha \in \mathcal A_{\rm{sc}}  \ &  \Longleftrightarrow  \ \exists \, k \notin \Lambda^*   \ \
\alpha^{-2} \in \Spec_{ L^2 ( \mathbb C/\Lambda ) } T_k \\
\ &   \Longleftrightarrow \  \forall \, k \notin \Lambda^*  \  \
\alpha^{-2} \in \Spec_{ L^2 ( \mathbb C/\Lambda ) } T_k  , \end{split} 
\end{equation}
Using the methods of \cite{bhz1} one can show that for $ V ( z ) = U ( z ) U ( - z )$ 
with $ U $ given by \eqref{eq:defU2} (or for more general classes of potentials described
in \cite{bhz1}),  $ \tr T_{k}^{p} \in ( \pi/\sqrt 3 ) \mathbb Q $, $ p \geq 2 $. Together
with a calculation for $ p = 2 $ (as in \cite{beta}) this
shows that $ | \mathcal A_{\rm{sc}} | =\infty $. With numerical assistance one can also show
existence of a real $ \alpha \in \mathcal A_{\rm{sc}} $.

\noindent 2. We can strengthen Proposition \ref{p:prot} as in \cite[Proposition 2.3]{bhz2}:
there exists a holomorphic family $ \mathbb C \ni \alpha \mapsto u ( \alpha ) \not \equiv 0 $, 
such that $  u ( 0 ) = 1 $ and $ Q ( \alpha, 0 ) u ( \alpha ) = 0 $. 

\section{Numerical observations} 

The spectral characterization \eqref{eq:spectral} allows for an accurate computation
of $ \alpha$'s for which \eqref{eq:defHV} exhibits flat bands at energy $ 0 $. 
For large $ \alpha$'s however, pseudospectral effects described in \cite{beta} 
make calculations unreliable. The set (shown as  \blue{$\bullet$}) 
 $ \mathcal A_{\rm{sc}} \cap \{ \Re \alpha \geq 0 \} $
where $\mathcal A_{\rm{sc}} $ is given in Theorem \ref{t:2} looks as follows
(for comparison we show the corresponding set, $ \mathcal{A}_{\rm{ch}}$, for the chiral model $\circ$):

\begin{center}
\includegraphics[width=15.5cm]{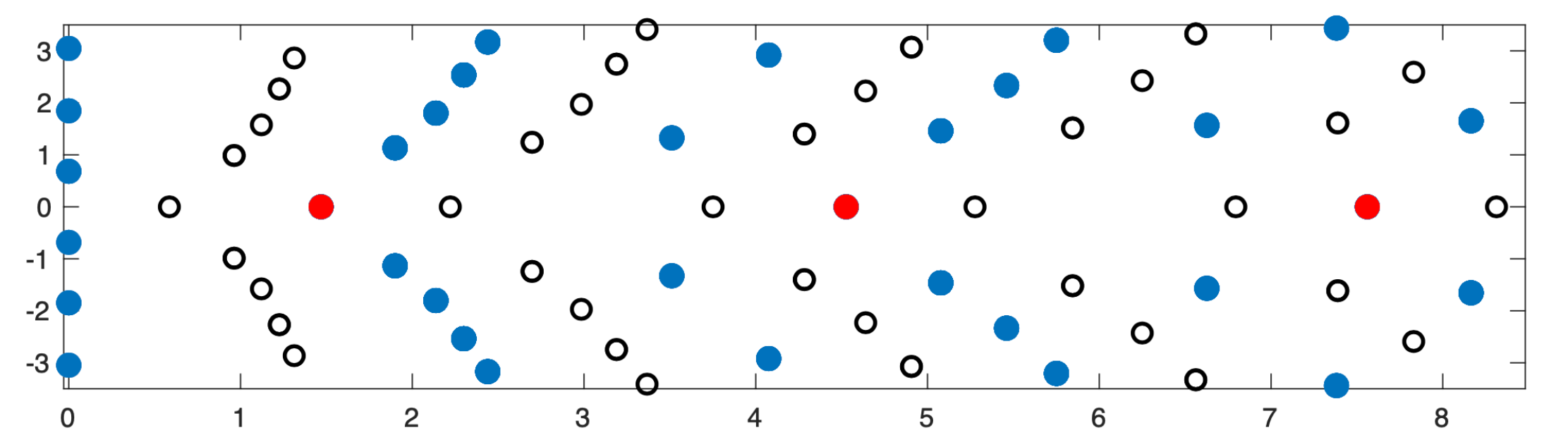}
\end{center} 

The real elements of $ \mathcal A_{\rm{sc}} $ are shown as \red{$\bullet$}. They appear to have multiplicity 
two. An adaptation of the theta function argument \cite{dun}, \cite{magic}, \cite{beta}, \cite[\S 3.2]{bhz2}
should apply to this case and the evenness of eigenfunctions in Proposition \ref{p:prot} shows that
they have (at least) two zeros at $ \alpha \in \mathcal A_{\rm{sc}}$. That implies multiplicity of at least $ 2 $.
This is illustrated by an animation \url{https://math.berkeley.edu/~zworski/scalar_magic.mp4} (shown
in the coordinates of \cite{beta}). When we interpolate between the chiral model and the 
scalar model, the multiplicity two real $ \alpha$'s  split and travel in opposite directions to become
magic $ \alpha $'s for the chiral model: see \url{https://math.berkeley.edu/~zworski/Spec.mp4}. 

One of the most striking observations made in \cite{magic} was a quantization rule for 
real elements of $ \mathcal A_{\rm{ch}} $ with the exact potential
\eqref{eq:defU}: if $ \alpha_1 < \alpha_2 < \cdots \alpha_j < \cdots $ is the 
 sequence of all real $ \alpha$'s for which \eqref{eq:flat} holds, then 
 \begin{equation}
 \label{eq:quant}  \alpha_{j+1} - \alpha_j  = \gamma + o ( 1 ) , \ \ j \to + \infty , \ \ \gamma \simeq \tfrac32.
 \end{equation}
The more accurate computations made in \cite{beta} suggests that 
$  \gamma \simeq 1.515 $. 

In the scalar model \eqref{eq:defHV} with $ V( z ) = U ( z ) U ( -z ) $ where 
$ U $ is given by \eqref{eq:defU} we numerically observe the following rule for real elements of $ \mathcal A_{\rm{sc}}$:
\begin{equation}
 \label{eq:quants}  \alpha_{j+1} - \alpha_j  = 2 \gamma + o ( 1 ) , \ \ j \to + \infty ,  \end{equation}
where $ \gamma $ is the same as in \eqref{eq:quant}. 

 \smallsection{Acknowledgements} 
We would like to thank Simon Becker for help with matlab and in particular for 
producing the movies referred to above. 
JG acknowledges support from EPSRC grants EP/V001760/1 and EP/V051636/1 and 
MZ from the NSF grant DMS-1901462
and the Simons Foundation under a ``Moir\'e Materials Magic" grant.

\end{document}